\documentclass{amsart}
\usepackage{amsthm,amssymb}
\usepackage{verbatim}
\usepackage{hyperref}
\usepackage{amscd,amsmath}

\usepackage[colorinlistoftodos]{todonotes} 

\newtheorem{theorem}{Theorem}
\newtheorem{proposition}[theorem]{Proposition}
\newtheorem{lemma}[theorem]{Lemma}

\newtheorem{corollary}[theorem]{Corollary}

\theoremstyle{definition}
\newtheorem{definition}[theorem]{Definition}

\newtheorem{remark}[theorem]{Remark}
\newtheorem{question}[theorem]{Question}

\makeatletter
\def\x@arrow{\DOTSB\Relbar}
\def\xlongequalsignfill@{\arrowfill@\x@arrow\Relbar\x@arrow}
\newcommand{\xlongequal}[2][]{%
        \ext@arrow 0099\xlongequalsignfill@{#1}{#2}}
\def\xlongrightarrowfill@{\arrowfill@\relbar\relbar\longrightarrow}
\newcommand{\xlongrightarrow}[2][]{%
        \ext@arrow 0099\xlongrightarrowfill@{#1}{#2}}

\newcommand{\6}{\partial}
\newcommand{\Ima}{{\text {Im}}\,}
\renewcommand{\ker}{\text {Ker}\,}
\renewcommand{\bar}{\overline}
\newcommand{\del}{\partial}
\newcommand{\delbar}{\partial_J}

\renewcommand{\phi}{\varphi}
\renewcommand{\[}{\begin{equation}}
\renewcommand{\]}{\end{equation}}

\begin{document}
\title{Quaternionic Bott--Chern cohomology and existence of HKT metrics.}
\author{Mehdi Lejmi}
\address{Department of Mathematics, Bronx Community College of CUNY, Bronx, NY 10453, USA.}
\email{mehdi.lejmi@bcc.cuny.edu}
\author{Patrick Weber}
\address{D\'epartement de Math\'ematiques, Universit\'e libre de Bruxelles CP218, Boulevard du Triomphe, Bruxelles 1050, Belgique.}
\email{pweber@ulb.ac.be}

\begin{abstract}
We study quaternionic Bott--Chern cohomology on compact hypercomplex manifolds and adapt some results from complex geometry to the quaternionic setting. For instance, we prove a criterion for the existence of HKT metrics on compact hypercomplex manifolds of real dimension~8 analogous to the one given by Teleman~\cite{tel} and Angella--Dloussky--Tomassini~\cite{ang-dlo-tom} for compact complex surfaces.
\end{abstract}

\maketitle

\vspace{-2em}

\section{Introduction}

A hypercomplex manifold $(M,I,J,K)$ is a manifold $M$ of real dimension $4n$ equipped with three complex structures $I,J,K$
satisfying the quaternionic relations. Such a manifold admits a unique torsion-free connection preserving $I,J,K$ called the Obata connection~\cite{oba}. If the holonomy of the Obata connection lies in the commutator subgroup $SL(n,\mathbb{H})$ of the general quaternionic linear group $GL(n,\mathbb{H})$, then the hypercomplex manifold $(M,I,J,K)$ is called an $SL(n,\mathbb{H})$-manifold.

From the metric point of view, we can always endow a hypercomplex manifold with a quaternionic Hermitian metric, i.e.~a Riemannian metric $g$ invariant with respect to all three complex structures $I,J,K.$ If we impose additional restrictions on the metric, we stumble upon some well-known metrics. If, for instance, we let $\Omega$ be the $(2,0)$-form with respect to $I$ defined by $\Omega(\cdot,\cdot):=g(J\cdot,\cdot)+\sqrt{-1}\,g(K\cdot,\cdot)$, then the metric $g$ is called hyperk\"{a}hler~\cite{bes} if $d\Omega=0$ and hyperk\"{a}hler with torsion~\cite{how-pap} (HKT for short) if $\partial \Omega =0,$ where $\partial$ is the Dolbeault operator with respect to $I$. 

G.~Grantcharov, M.~Verbitsky and the first named author recently proved that a compact $SL(2, {\Bbb H})$-manifold is HKT if and only if the Hodge number~$h^{0,1}$ (with respect to $I$) is even~\cite{gra-lej-ver}. This is the quaternionic counterpart of the well-known result that a compact complex surface is K\"ahler if and only if the first Betti number is even~\cite{buchdahl,lamari,miy, siu}. The aim of the present paper is to further investigate this link between compact $SL(2,\mathbb{H})$-manifolds with HKT metrics on one hand and compact complex surfaces with K\"{a}hler metrics on the other hand.  We adapt the following complex geometry results  to the quaternionic setup:

\hfill

\begin{minipage}{0.45\textwidth}
A compact complex surface
\begin{enumerate}
\item{has its Fr\"{o}licher spectral sequence degenerate at the first page~\cite{kodaira}.}
\item{is K\"{a}hler if and only if the second non-K\"{a}hler-ness degree vanishes~\cite{tel}.}
\item{is K\"{a}hler if and only if it is strongly Gauduchon~\cite{pop}.\\}
\end{enumerate}
\end{minipage}%
\hfill
\begin{minipage}{0.45\textwidth}
A compact $SL(2,\mathbb{H})$-manifold
\begin{enumerate}
\item{has its quaternionic Fr\"{o}licher spectral sequence degenerate at the first page. (Section 8)}
\item{is HKT if and only if the second non-HKT-ness degree vanishes. (Section~9)}
\item{is HKT if and only if it is quaternionic strongly Gauduchon. (Section 10)}
\end{enumerate}
\end{minipage}

\hfill

The fundamental difference from the complex setup is that we are not in the presence of a double complex. Instead we are facing a single complex with two anti-commuting differentials acting on the same component of the bidegree. The upshot is that most proofs simplify, but on the downside we sometimes need to impose additional assumptions, such as the $SL(n,\mathbb{H})$ holonomy constraint, to recover the necessary symmetries.

\hfill

The organisation of the paper is as follows. In Section~\ref{section2} and~\ref{section3} we recall the definition of quaternionic cohomology groups and give quaternionic versions of the Fr\"{o}licher spectral sequence and of the $dd^c$-Lemma. In Section~\ref{section4} we state a Fr\"{o}licher-type inequality for quaternionic Bott--Chern cohomology on compact hypercomplex manifolds as proven by
Angella--Tomassini in a very general setup in~\cite{ang-tom-1} and for compact complex manifolds in~\cite{_Angella-Tomassini_}. In Section~\ref{section5} we take a look at quaternionic metric structures and in Section~\ref{section6} we build a Hodge theory in the particular case when the holonomy of the Obata connection is contained in $SL(n,\mathbb{H})$. In Section~\ref{section7} we show how to obtain a type decomposition in real and imaginary $(2,0)$-forms by adapting tools from almost-complex geometry~\cite{dra-li-zha}. In Section~\ref{froelicherSL} we show that the quaternionic spectral sequence degenerates at the first page on compact $SL(2,\mathbb{H})$-manifolds. In Section~\ref{section9} and~\ref{section10} we prove that a compact $SL(2,\mathbb{H})$-manifold admits a HKT metric if and only if it admits a quaternionic strongly Gauduchon metric, and that this happens precisely when the second non-HKT-ness degree vanishes. Finally, in Section~\ref{section11} we compute the quaternionic cohomology groups of some nilmanifolds in real dimension 8 and 12.

%
%
%


\section{Quaternionic cohomology groups}\label{section2}


A {\it hypercomplex manifold} $(M,I,J,K)$ is a smooth manifold $M$ of real dimension~$4n$ equipped with
three integrable almost-complex structures $I, J, K$
satisfying the quaternionic relations
$I^2=J^2=K^2= IJK = -\text{Id}_{TM}$.
We denote by $\Omega^\ast(M)=\sum \limits_{i=0}^{4n}\Omega^i(M)$ the bundle of real exterior forms on $M$ and by
$$\Omega^i(M)\otimes\mathbb{C}=\underset{p+q=i}{\displaystyle\bigoplus}\Omega_{I}^{p,q}(M),$$
the decomposition of the complexified bundle with respect to~$I.$ The sections of~$\Omega_I^{p,q}(M)$ are denoted by $\Lambda_I^{p,q}(M).$
The complex structure $J$ acts on $1$-forms and this action is extended to the space of $p$-forms by
$$(J\varphi)(X_1,\cdots,X_p)=\varphi(J^{-1}X_1,\cdots,J^{-1}X_p)=(-1)^p\varphi(JX_1,\cdots,JX_p),$$
for any $p$-form $\varphi.$
Since $I$ and $J$ anti-commute, the action by $J$ matches $(p,q)$-forms with respect to $I$ to $(q,p)$-forms:
$$J:\Lambda^{p,q}_I(M)\longrightarrow\Lambda^{q,p}_I(M).$$
We define the twisted differential operator
$\6_J:  \Lambda^{p,q}_I(M)\rightarrow \Lambda^{p+1,q}_I(M)$
by
$$\6_J:=J^{-1}\circ \bar\6 \circ J,$$
where $\bar{\partial}$ is the standard Dolbeault operator with respect to $I$.
Note that both $\partial$ and $\partial_J$ increase $p$ in the bidegree $(p,q)$ and leave $q$ untouched. By the integrability of~$I$, we deduce $\partial^2=\6_J^2=0$.
Moreover, by a result of~\cite{_Verbitsky:HKT_}, the operators $\6$ and $\6_J$ anti-commute: 
$$\6\6_J+\6_J\6=0.$$ If we fix $q=0$, we get a (single) cochain complex $(\Lambda^{p,0}_I(M), \partial, \partial_J)$ with two anti-commuting differentials.
%
%
%
%
%
For a hypercomplex manifold $(M,I,J,K)$ of real dimension $4n$, the complex $(\Lambda^{p,0}_I(M),\partial,\partial_J)$ leads to the \textsl{quaternionic Dolbeault cohomology groups} 
\begin{eqnarray*}
H_{\partial}^{p,0} &:=& \frac{\{ \phi\in {\Lambda}^{p,0}_I(M)\, |\, \6 \phi =0\}}{\6{\Lambda}^{p-1,0}_I(M)} = \frac{\ker \partial}{\Ima \partial},\\
H^{p,0}_{\partial_J} &:=& \frac{\{ \phi\in {\Lambda}_I^{p,0}(M)\, |\, \6_J\phi= 0\}}{\6_J{\Lambda}^{p-1,0}_I(M)} = \frac{\ker \partial_J}{\Ima \partial_J}.
\end{eqnarray*}
Moreover, one can define the \textsl{quaternionic Bott--Chern cohomology groups} 
\begin{eqnarray*}
H^{p,0}_{BC} := \frac{\{ \phi\in
  {\Lambda}^{p,0}_I(M)\, |\, \6 \phi =\6_J\phi=
  0\}}{\6\6_J{\Lambda}^{p-2,0}_I(M)} = \frac{\ker \partial \cap \ker \partial_J}{\Ima \partial \partial_J},
\end{eqnarray*} 
and the \textsl{quaternionic Aeppli cohomology groups} 
\begin{equation*}
H^{p,0}_{AE} := \frac{\{ \phi\in
  {\Lambda}_I^{p,0}(M)\, |\, \6\6_J\phi= 0\}}{\6{\Lambda}^{p-1,0}_I(M)+\6_J{\Lambda}^{p-1,0}_I(M)}= \frac{\ker \partial \partial_J}{\Ima \partial + \Ima \partial_J}.
\end{equation*} 
It has been shown in~\cite{gra-lej-ver} that all of these groups are finite-dimensional:

\begin{theorem}\cite{gra-lej-ver}
On a compact hypercomplex manifold, the groups $H^{p,0}_{BC}$ and $H^{p,0}_{AE}$ are finite-dimensional.
\end{theorem}
To ease the notation, we use $h_{BC}^{p,0}=\dim H_{BC}^{p,0}$, $h^{p,0}_{\delbar}=\dim H^{p,0}_{\delbar}$, etc.
Following J.~Varouchas~\cite{_Varouchas_}, we define the complex vector spaces
$$A^{p,0}:= \frac{\Ima \del \cap \Ima \delbar}{\Ima \del\delbar}, \quad B^{p,0}:= \frac{\ker \6 \cap \Ima \6_J}{\Ima \del\delbar}, \quad C^{p,0}:= \frac{\ker \del\delbar}{\ker \6 + \Ima \6_J},$$
$$D^{p,0}:= \frac{\Ima \6 \cap \ker \6_J}{\Ima \del\delbar}, \quad E^{p,0}:= \frac{\ker \del\delbar}{\Ima\6 + \ker \6_J}, \quad F^{p,0}:= \frac{\ker \del\delbar}{\ker \del + \ker \delbar},$$
where $\6$ and $\6_J$ are acting on forms of type $\Lambda^{p,0}_I(M)$ with $0\leqslant p \leqslant 2n$.
Once more the dimensions of these spaces will be denoted by the corresponding lower case letter, for example $a^{p,0}=\dim A^{p,0}$.

\begin{lemma}\label{Var_sequence}
The two sequences
\begin{eqnarray*}
0 \longrightarrow A^{p,0} \longrightarrow B^{p,0} \longrightarrow H_{\partial}^{p,0} \longrightarrow H_{AE}^{p,0} \longrightarrow C^{p,0} \longrightarrow 0
\end{eqnarray*}
\begin{eqnarray*}
0 \longrightarrow D^{p,0} \longrightarrow H_{BC}^{p,0} \longrightarrow H_{\partial}^{p,0} \longrightarrow E^{p,0} \longrightarrow F^{p,0} \longrightarrow 0
\end{eqnarray*}
are exact sequences of finite-dimensional vector spaces.
\end{lemma}

\begin{proof}
Exactness only relies on $\partial^2=0$ whereas finiteness can be deduced from finiteness of $h_{BC}^{p,0}$ and $h_{AE}^{p,0}$.
\end{proof}

As in the complex setup, Varouchas' short exact sequences will help us to quantify the discrepancy between the quaternionic Dolbeault groups and the quaternionic Bott--Chern/Aeppli groups. Last, but not least, the following map will be a crucial tool in the analysis of quaternionic cohomology groups.

\begin{proposition}\label{conjugation-symmetry}
On a hypercomplex manifold, the map
\begin{eqnarray*}
\bar{J} : \Lambda^{p,0}_I(M)&\longrightarrow&\Lambda^{p,0}_I(M),\\
\phi&\longmapsto& J(\bar{\phi}),
\end{eqnarray*}
induces the isomorphisms
\begin{eqnarray*}
H^{p,0}_{\partial} \simeq H^{p,0}_{\partial_J}, \quad B^{p,0} \simeq D^{p,0} \quad {\text{ and }} \quad C^{p,0}\simeq E^{p,0}.
\end{eqnarray*}
Here $\bar{\phi}$ denotes the complex conjugate of $\phi$ with respect to $I$. Moreover, 
the maps
\begin{equation*}
E^{p,0} \xlongrightarrow{\partial_J} B^{p+1,0}  \quad {\text{ and }}\quad C^{p,0} \xlongrightarrow{\partial} D^{p+1,0}
\end{equation*}
are isomorphisms.
\end{proposition}


\section{Quaternionic Fr\"{o}licher spectral sequence and $dd^c$-Lemma}\label{section3}


In this Section we introduce a quaternionic analog of the Fr\"{o}licher spectral sequence 
and we show that degeneration at the first page is a weaker condition than the quaternionic $dd^c$-Lemma introduced in~\cite{gra-lej-ver}. For more information about the Fr\"{o}licher spectral sequence in the complex setup, we refer to~\cite{froelicher}.\\

The differential $\partial_J$ induces homomorphisms between the cohomology groups 
$$\cdots \xlongrightarrow{\partial_J} H^{p-1,0}_\partial \xlongrightarrow{\partial_J} H^{p,0}_\partial \xlongrightarrow{\partial_J} H^{p+1,0}_\partial \xlongrightarrow{\partial_J} \cdots.$$
As $\partial_J^2=0$, we may ask whether the above sequence has any chance to be exact. Even if this fails, we can still try to quantify the deviation from exactness. 
For this purpose, we define the complex vector spaces:
\begin{eqnarray*}
E_1^{p,0} &:=& H^{p,0}_\partial = \frac{ \{ \varphi \in \Lambda^{p,0}_I(M)~|~\partial \varphi = 0\} }{\6\Lambda_{I}^{p-1,0}(M)},\\
E_2^{p,0} &:=& \frac{ \{ \varphi \in \Lambda^{p,0}_I(M)~|~\partial \varphi = 0 \text{ and } \partial_J \varphi = \partial \alpha\} }{\{\varphi \in \Lambda^{p,0}_I(M)~|~\varphi = \partial \beta + \partial_J \gamma \text{ and } \partial \gamma =0\}}.
\end{eqnarray*}
On a compact hypercomplex manifold, these vector spaces are finite-dimensional as $\dim E^{p,0}_2 \leqslant \dim E_1^{p,0}=h^{p,0}_\partial < \infty$. Note that, if
the maps $H^{p,0}_\partial \xlongrightarrow{\partial_J} H^{p+1,0}_\partial$ vanish for all $p$, then $E_1^{p,0} \simeq E_2^{p,0}$ for all $p$.
The groups $E_2^{p,0}$ measure the deviation from exactness of the above sequence. 
As in the present work, we just need the groups~$E_1^{p,0}$ and $E_2^{p,0}$, we may define the groups $E_{r+1}^{p,0}$ as the cohomology groups of the sequence
$$\cdots \xlongrightarrow{\partial_J} E_r^{p-1,0} \xlongrightarrow{\partial_J} E_r^{p,0} \xlongrightarrow{\partial_J} E_r^{p+1,0} \xlongrightarrow{\partial_J} \cdots.$$
As in the complex setup, we call $\{E_r^{p,0}\}_{p\in \mathbb{N}_0}$ the \textsl{$r^{\text{th}}$ page of the quaternionic Fr\"{o}licher spectral sequence}. The quaternionic Fr\"{o}licher spectral sequence is said to \textsl{degenerate at the $r^{\text{th}}$ page} if $E^{p,0}_r \simeq E^{p,0}_{r+1}$ for all $p$ (note that then also $E_{r}^{p,0} \simeq E_{r+i}^{p,0}$ for all~$p$ and all $i\geqslant 0$). 
The next Lemma gives a sufficient condition for the quaternionic Fr\"{o}licher spectral sequence to degenerate at the first page:

%
\begin{lemma}\label{sufficientnew}
If $A^{p+1,0} \simeq B^{p+1,0}$, then $E_1^{p,0} \simeq E_2^{p,0}$.
\end{lemma}

\begin{proof}
Let $\mathfrak{a} \in E_1^{p,0}$ be represented by $\varphi$. Then $\partial_J \varphi$ lies in $B^{p+1,0}$. By hypothesis $A^{p+1,0} \simeq B^{p+1,0}$, so it also sits inside $A^{p+1,0}$ and hence is $\partial$-exact. 
\end{proof}
However, this condition is not necessary as can be seen from Example 1 in Section~\ref{section11} where we have $a^{3,0} \neq b^{3,0}$ but nevertheless $E_1^{2,0} \simeq E_2^{2,0}$.
In~\cite{gra-lej-ver} a quaternionic version of the $dd^c$-Lemma has been introduced:
\begin{definition}
Let $(M,I,J,K)$ be a hypercomplex manifold of real dimension $4n$. We say that the $\partial \partial_J$-Lemma holds
if every $\partial$-closed, $\partial_J$-exact $(p,0)$-form in $\Lambda^{p,0}_I(M)$ is $\partial \partial_J$-exact, for
any $0\leqslant p \leqslant 2n$. This is equivalent to saying that every $\partial_J$-closed, $\partial$-exact form is $\partial \partial_J$-exact.
\end{definition}

Note that, by definition, the $\partial \partial_J$-Lemma holds for $(p,0)$-forms if and only if the vector space $B^{p,0}$ (or equivalently $D^{p,0}$) is trivial. In order to relate the $\partial \partial_J$-Lemma to the quaternionic Fr\"{o}licher spectral sequence, we need the following result:

\begin{lemma}\label{frolicher-implies-Lemma}
If $A^{p+1,0} \simeq 0$ and $E_1^{p,0} \simeq E_2^{p,0}$, then the map $H_{BC}^{p,0} \longrightarrow H_{\partial}^{p,0}$ is surjective, i.e.~$E^{p,0} \simeq 0$.
\end{lemma}

\begin{proof}
Let $\mathfrak{a} \in H^{p,0}_{\partial}$. We need to find a representative $\varphi $ such that
$\partial_J \varphi =0.$
The assumption $E_1^{p,0} \simeq E_2^{p,0}$ implies that there exists some $\alpha$ such that $\partial_J \varphi = \partial \alpha$. Hence $\partial_J \varphi $ is both $\partial$-exact and $\partial_J$-exact. Finally, the hypothesis $A^{p+1,0} \simeq 0$ implies that $\partial_J \varphi$ lies in the Image of $\partial \partial_J$.
\end{proof}

This leads to the following relation between the $\partial \partial_J$-Lemma and the quaternionic Fr\"{o}licher spectral sequence:

\begin{theorem}\label{froh-lemma}
$B^{p+1,0} \simeq 0$ if and only if both $A^{p+1,0} \simeq 0$ and $E_1^{p,0} \simeq E_2^{p,0}$.
\end{theorem}

\begin{proof}
If $B^{p+1,0} \simeq 0$, then $A^{p+1,0} \simeq 0$ by exactness of Varouchas'~sequence. Lemma~\ref{sufficientnew} then implies that $E^{p,0}_1 \simeq E^{p,0}_2$. Conversely, we may apply Lemma~\ref{frolicher-implies-Lemma} to deduce that $E^{p,0} \simeq 0$. The isomorphism $\partial_J : E^{p,0} \to B^{p+1,0}$ from Proposition~\ref{conjugation-symmetry} finishes the proof.
\end{proof}

In words, if the $\partial \partial_J$-Lemma holds, then the quaternionic Fr\"{o}licher spectral sequence degenerates at the first page. However, in order to get the implication in the opposite direction, we need the additional assumption $A^{p+1,0} \simeq 0$. 
\begin{remark}
Example~1 in Section~\ref{section11} provides an example such that the quaternionic Fr\"{o}licher spectral sequence degenerates at the first page but such that the $\partial \partial_J$-Lemma does not hold. 
\end{remark}
We will return to the quaternionic Fr\"{o}licher spectral sequence in Section~\ref{froelicherSL} when we show that the quaternionic Fr\"{o}licher spectral sequence always degenerates at the first page in real dimension $8$ on manifolds with special holonomy, called $SL(2,\mathbb{H})$-manifolds.


\section{Quaternionic Fr\"{o}licher-type inequality}\label{section4}



In this Section, we give a quaternionic version of the Fr\"{o}licher-type inequality proven by D.~Angella and A.~Tomassini in~\cite{ang-tom-1} in a very general setting (see also~\cite{_Angella-Tomassini_}). As in~\cite{ang-tom-1}, we deduce a cohomological criterion for the quaternionic $dd^c$-Lemma. Since $\left({\Lambda}^{p,0}_I(M),\6,\6_J\right)$ is not a double complex, part of our proof differs slightly. For the reader's convenience, we give the details.
%

\begin{theorem}\cite{ang-tom-1}\label{froh_global}
Let $(M,I,J,K)$ be a compact hypercomplex manifold of real dimension $4n$. Then
\begin{equation*}
h_{BC}^{p,0}+h_{AE}^{p,0} \geqslant 2~h^{p,0}_{\partial} \geqslant 2 \dim E^{p,0}_2
\end{equation*}
for any $0\leqslant p \leqslant 2n.$
\end{theorem}

\begin{proof}
By exactness of Varouchas' short exact sequences, we have
\begin{equation*}
a^{p,0} - b^{p,0} + h_{\partial}^{p,0} - h_{AE}^{p,0} + c^{p,0}  =0,
\end{equation*}
\begin{equation*}
d^{p,0} - h_{BC}^{p,0} + h_{\partial}^{p,0} - e^{p,0} + f^{p,0}  =0.
\end{equation*}
Summing up both we get,
\begin{eqnarray*}
h_{BC}^{p,0} + h_{AE}^{p,0} &=& 2 h_{\partial}^{p,0} + a^{p,0} + f^{p,0} + \left(c^{p,0}-e^{p,0}\right) + \left( d^{p,0}-b^{p,0} \right)\\
&=& 2 h_{\partial}^{p,0} + a^{p,0} + f^{p,0} \\
& \geqslant& 2 h_{\partial}^{p,0} \\
& \geqslant& 2 \dim E_2^{p,0}.
\end{eqnarray*}
The second equality follows from Proposition~\ref{conjugation-symmetry}.
\end{proof}

The following Theorem shows that equality happens precisely when the quaternionic $dd^c$-Lemma is satisfied.

\begin{theorem}\label{_bott_chern_aeppli_equa_}\cite{_Angella-Tomassini_,ang-tom-1} 
Let $(M,I,J,K)$ be a compact hypercomplex manifold of real dimension $4n.$
Then the $\6\6_J$-Lemma holds if and only if the equality
$$h_{BC}^{p,0} + h_{AE}^{p,0}= 2 \dim E_2^{p,0}$$
holds for all $0\leqslant p \leqslant 2n.$
\end{theorem}

\begin{proof}
If we assume the equality to hold, then the proof of Theorem~\ref{froh_global} shows that both $E_1^{p,0} \simeq E_2^{p,0}$ for all $p$ and $a^{p,0}=0=f^{p,0}$ for all $p$. Theorem~\ref{froh-lemma} implies that $B^{p,0} \simeq 0$ for all $p$ and hence the $\partial \partial_J$-Lemma holds. Conversely, if the $\partial \partial_J$-Lemma holds, then $b^{p,0}=0=d^{p,0}$ for all $p$. Exactness of Varouchas' sequences implies that $a^{p,0}=0=f^{p,0}$ and by Lemma~\ref{sufficientnew} we also get $E_1^{p,0} \simeq E_2^{p,0}$.
\end{proof}

\section{Metric structures}\label{section5}


Up to this stage, the results have been cohomological in nature. We will now start to endow the hypercomplex manifolds with additional geometric structures. Recall the following special Hermitian metrics defined on complex manifolds, see for example~\cite{pop}:
$$\text{Gauduchon} \xleftarrow{} \text{strongly Gauduchon} \xleftarrow{} \text{K\"{a}hler}.$$
In this short Section, we introduce the quaternionic counterparts of these metrics. A \textsl{quaternionic Hermitian metric} on a hypercomplex manifold $(M,I,J,K)$ is a Riemannian metric~$g$ satisfying
\begin{equation*}
g(\cdot,\cdot)=g(I\cdot,I\cdot)=g(J\cdot,J\cdot)=g(K\cdot,K\cdot).
\end{equation*}
Given a quaternionic Hermitian metric $g$ on $(M,I,J,K)$, we have
the following differential forms
\begin{equation*}\label{_three_2_forms_qH_Equation_}
\omega_I(\cdot, \cdot) := g(I\cdot, \cdot), \ \ \omega_J(\cdot, \cdot) := g(J\cdot, \cdot)
, \ \ \omega_K(\cdot, \cdot) := g(K\cdot, \cdot).
\end{equation*}
The form $$\Omega:= \omega_J + \sqrt{-1}\omega_K$$
is of Hodge type $(2,0)$ with respect to the complex structure $I$, i.e.~$\Omega \in \Lambda^{2,0}_I(M).$ Moreover, it is non-degenerate. Actually we can reconstruct $g$ from~$\Omega$. Indeed, for any vector fields $X,Y$ of type $T^{1,0}_I(M)$ we have
$$g(X,\bar{Y})=\frac{1}{2}\Omega(X,J\bar{Y}).$$

\begin{definition}
Let $(M,I,J,K,\Omega)$ be a quaternionic Hermitian manifold of real dimension~$4n$.
\begin{itemize}
\item{If $\partial \partial_J \Omega^{n-1} =0$, then the metric is called \textsl{quaternionic Gauduchon} (see~\cite{gra-lej-ver}). This is the quaternionic analog of the Hermitian Gauduchon metric~\cite{gau}.}
\item{If $\partial \Omega^{n-1}$ is $\partial_J$-exact, then the metric is called \textsl{quaternionic strongly Gauduchon}. This is the quaternionic analog of the Hermitian strongly Gauduchon metric~\cite{pop}.}
\item{If $\partial \Omega =0$, then the metric is called \textsl{hyperk\"{a}hler with torsion}, HKT for short~\cite{how-pap}. For more about HKT metrics see~\cite{gra-poon}.}
\item{If $d\Omega=0$, then the metric is called \textsl{hyperk\"{a}hler}~\cite{bes}.}
\end{itemize} 
\end{definition}
One checks that
$$\text{quaternionic Gauduchon} \xleftarrow{} \text{quaternionic strongly Gauduchon} \xleftarrow{} \text{HKT}.$$
An important feature in non-K\"{a}hler complex geometry is that a complex manifold always admits a Hermitian Gauduchon metric. In the next Section we will see that, if we impose a constraint on the holonomy, then we similarly get existence of a quaternionic Gauduchon metric on hypercomplex manifolds. 



\section{$SL(n,\mathbb{H})$-manifolds}\label{section6}


First, we recall that a hypercomplex manifold admits a canonical connection. 

\begin{theorem}\cite{oba}
A hypercomplex manifold $(M,I,J,K)$ admits a unique torsion-free connection preserving all three complex structures $I,J,K.$ This connection is called the Obata connection.
\end{theorem}

Since the Obata connection preserves the hypercomplex structure, its holonomy lies in the general quaternionic linear group $GL(n, {\Bbb H})$ (see for instance~\cite{_Soldatenkov:SU(2)_}).
However, in many examples such as nilmanifolds~\cite{_BDV:nilmanifolds_}, the holonomy is actually contained in $SL(n, {\Bbb H})$, the commutator subgroup of $GL(n,\Bbb H)$
(for more details see~\cite{_Verbitsky:canoni_,_Verbitsky:HKT_}).

\begin{definition}
A hypercomplex manifold of real dimension $4n$ is called an $SL(n, {\Bbb H})$-manifold if the holonomy of the Obata
connection lies in $SL(n, {\Bbb H})$.
\end{definition}

For any $SL(n, {\Bbb H})$-manifold, the
canonical bundle~$\Omega^{2n,0}_I(M)$ is holomorphically trivial~\cite{_Verbitsky:canoni_}.
The converse is true if, moreover, the manifold is compact and admits an HKT metric~\cite{_Verbitsky:canoni_}.
We denote an $SL(n, {\Bbb H})$-manifold by $(M,I,J,K,\Phi)$ where
$\Phi$ is a nowhere degenerate 
holomorphic section of $\Omega^{2n,0}_I(M)$. Furthermore, we can assume that $\Phi$ satisfies
$\Phi=J\bar\Phi$. In particular, $\6\bar{\Phi}=\6_J\bar{\Phi}=0.$\\

Whilst HKT metrics or quaternionic strongly Gauduchon metrics do not always exist on hypercomplex manifolds (see~\cite{fin-gra,swa} and Remark~\ref{non-existenceSG}) the following result shows that $SL(n, \mathbb{H})$-manifolds always admit a quaternionic Gauduchon metric.

\begin{proposition}\cite{gra-lej-ver}
Let $(M,I,J,K,\Phi,\Omega)$ be an $SL(n,\Bbb H)$-manifold endowed with a quaternionic Hermitian metric $g$. Then there exists a unique (up to constant) positive function~$\mu$ such that $\mu g$ is quaternionic Gauduchon.
\end{proposition}

This crucial result is a first reason why in the sequel we will always assume to be on an $SL(n,\mathbb{H})$-manifold. The second reason is that it is possible to build a version of Hodge theory on $SL(n,\mathbb{H})$-manifolds~\cite{_Verbitsky:HKT_} as shown next.

Let $(M,I,J,K,\Phi,\Omega)$ be a compact $SL(n,\Bbb H)$-manifold endowed with a quaternionic Hermitian metric $g$. We can suppose that $\Phi$ is positive (compared to $\Omega^n$)
and satisfies $\Phi=J\bar\Phi.$ Let $\{e_1,e_2=J\bar{e}_1,\cdots,{e}_{2n-1},e_{2n}=J\bar{e}_{2n-1}\}$ be a local quaternionic $h$-orthonormal basis of $\Lambda^{1,0}_I(M)$ (here $2\,h(X,Y)=g(X,\bar{Y})-\sqrt{-1}g(IX,\bar{Y})$, where $X,Y$ are sections of $T(M)\otimes\mathbb{C}$).
Locally we can write~\cite[Proposition~3.2]{_Verbitsky:holom} $$\Omega=\sum_{i=1}^n e_{2i-1}\wedge e_{2i}.$$
Similarly to~\cite[Section 9]{_Verbitsky:HKT_}, we define the Hodge-star operator  $$\ast_{\Phi} : \Lambda^{p,0}_I(M) \longrightarrow \Lambda^{2n-p,0}_I(M)$$ in the following way:
\begin{equation*}
\alpha \wedge (\ast_\Phi \beta)  \wedge \bar{\Phi} = h(\alpha, {\beta}) \,\frac{\Omega^n \wedge \bar{\Phi}}{n!},
\end{equation*}
for $\alpha,\beta\in\Lambda^{p,0}_I(M)$ ($h$ is the induced metric on $\Lambda^{p,0}_I(M)$, in particular, $h(\Omega,\Omega)=n$ and $h(\Omega^n,\Omega^n)=(n!)^2$). 
On multi-vectors, the operator $\ast_{\Phi}$ acts as follows
$$\ast_{\Phi}\left(e_1\wedge\cdots\wedge e_p\right)=e_{p+1}\wedge\cdots\wedge e_{2n}.$$
One checks that $\left(\ast_\Phi\right)^2 \alpha=(-1)^p\,\alpha$
for $\alpha\in\Lambda^{p,0}_I(M)$.\\

The operator $\6^{\ast_{\Phi}}$ (resp.~$\6_J^{\ast_{\Phi}}$) is defined as the adjoint of $\6$ (resp.~$\6_J$) with respect to the (global) Hermitian inner product
\begin{eqnarray}\label{_usual_prod_}
\langle\alpha,\beta\rangle=\int_M h(\alpha,\beta)\,\frac{\Omega^n\wedge\bar{\Phi}}{n!}.
\end{eqnarray}
It is a straightforward computation to deduce that 
\begin{equation}\label{hodge-adjoints}
\6^{\ast_{\Phi}}=-\ast_{\Phi}\6\,\ast_{\Phi},\quad\6_J^{\ast_{\Phi}}=-\ast_{\Phi}\6_J\,\ast_{\Phi}.
\end{equation}
\begin{remark}
Here we used that $\int_M\6\phi=\int_M\6_J\phi=0$ for any $\phi\in\Lambda^{2n-1,2n}_I(M).$ The vanishing of $\int_M\6\phi$ follows from Stokes's theorem since $\6\phi=d\phi$,
where $d$ is the exterior derivative.
To prove $\int_M\6_J\phi=0$, we use the fact that $\Omega^n\wedge\bar\Omega^n$ is a volume form on $M$ and that $J\bar\Omega=\Omega$ (in particular, $J\left(\Omega^n\wedge\bar\Omega^n\right)=\Omega^n\wedge\bar\Omega^n$). Indeed,
$\int_M\6_J\phi=\int_Mg(\6_J\phi,\Omega^n\wedge\bar\Omega^n)\,\Omega^n\wedge\bar\Omega^n=\int_Mg(\bar\6J\phi,\Omega^n\wedge\bar\Omega^n)\,\Omega^n\wedge\bar\Omega^n=\int_M\bar\6J\phi=0$
since $\bar\6J\phi=dJ\phi.$
\end{remark}

If the dimension of the manifold is $4n=8$, then the bundle $\Omega^{2,0}_I(M)$ admits an ($h$-orthogonal) decomposition
\begin{equation}\label{dual_decomposition}
\Omega^{2,0}_I(M)=\Omega^{{\Phi},+}(M)\oplus\Omega^{{\Phi},-}(M),
\end{equation}
where $\Omega^{{\Phi},+}(M)$ (resp.~$\Omega^{{\Phi},-}(M)$) is the bundle of $\ast_{\Phi}$-self-dual forms $\phi\in\Lambda^{2,0}_I(M)$ (resp.~$\ast_{\Phi}$-anti-self-dual forms).
A section $\varphi$ of $\Omega^{{\Phi},+}(M)$ satisfies $\ast_{\Phi}\phi=\phi$ and is called \textsl{$\ast_{\Phi}$-self-dual} whereas a section $\varphi$ of $\Omega^{{\Phi},-}(M)$ satisfies $\ast_{\Phi}\phi=-\phi$ and is called \textsl{$\ast_{\Phi}$-anti-self-dual}. We deduce from~\eqref{hodge-adjoints} that the Hodge star operator $\ast_{\Phi}$ commutes with
the $\6$-Laplacian $\Delta_{\6}=\6\6^{\ast_{\Phi}}+\6^{\ast_{\Phi}}\6.$ Hence the decomposition~\eqref{dual_decomposition} descends to cohomology and the space $H_{\6}^{2,0}$ can be written as a direct sum of $\6$-closed $\ast_{\Phi}$-self-dual and $\6$-closed $\ast_{\Phi}$-anti-self-dual forms:
\begin{equation*}
H^{2,0}_\partial = H^{{\Phi},+}_\partial \oplus H^{{\Phi},-}_\partial.
\end{equation*}
Note that, more generally, on $SL(n,\mathbb{H})$-manifolds with $n$ even, a similar decomposition can be done for the space $H^{n,0}_\partial$.

Finally, when $(M,I,J,K,\Phi)$ is a compact $SL(n,\mathbb{H})$-manifold, the groups $H^{p,0}_{BC}$ and $H^{2n-p,0}_{AE}$ are dual. Indeed, consider the pairing~\cite{gra-lej-ver} on $H^{p,0}_{BC}\times H^{2n-p,0}_{AE}$ given by
\begin{equation}\label{pairing}
\left([\alpha],[\beta]\right)\longmapsto \int_M\alpha\wedge\beta\wedge\bar{\Phi}.
\end{equation}
Using the same pairing, one also deduces that 
\begin{equation*}
H_{\6}^{p,0}\simeq H_{\6}^{2n-p,0} \text{ and } H_{\6_J}^{p,0}\simeq H_{\6_J}^{2n-p,0}.
\end{equation*}
We will refer to these isomorphisms as \textsl{$SL(n,\mathbb{H})$-symmetry}.
%


\section{Pure and full hypercomplex structures}\label{section7}


Besides the decomposition in self-dual and anti-self-dual forms, it is also possible to decompose $(2,0)$-forms in real and imaginary forms. Indeed, any $\phi\in{\Lambda}^{2,0}_I(M)$ can be written as $\phi=\phi^{\bar{J},+}+\phi^{\bar{J},-}$, where
$$\phi^{\bar{J},+}:=\frac{1}{2}\left(\phi+J\bar{\phi}\right) \quad \text{ and } \quad \phi^{\bar{J},-}:=\frac{1}{2}\left(\phi-J\bar{\phi}\right).$$ We get a decomposition of the bundle $\Omega^{2,0}_I(M)$:
\begin{equation}\label{real_decomposition}
\Omega^{2,0}_I(M)=\Omega^{\bar{J},+}(M)\oplus\Omega^{\bar{J},-}(M)
\end{equation}
A section $\phi$ of $\Omega^{\bar{J},+}(M)$ satisfies $\phi=J\bar{\phi}$ and is called \textsl{real} whereas a section $\varphi$ of $\Omega^{\bar{J},-}(M)$ satisfies $\phi=-J\bar{\phi}$ and is called \textsl{imaginary}.\\

On compact $SL(2,\mathbb{H})$-manifolds, we have the following relation between the two decompositions~(\ref{dual_decomposition}) and~(\ref{real_decomposition}): 
\begin{lemma}
On a compact $SL(2,\Bbb H)$-manifold $(M,I,J,K,\Phi,\Omega)$ with quaternionic Hermitian metric we have
\begin{equation}\label{two_decompositions}
\Omega^{{\Phi},+}(M)=C^\infty(M) \wedge \Omega\oplus \Omega^{\bar{J},-}(M),\quad \Omega^{{\Phi},-}(M)=\Omega_0^{\bar{J},+}(M),
\end{equation}
where $\Omega_0^{\bar{J},+}(M)$ is the bundle of real forms pointwise $h$-orthogonal to $\Omega.$
\end{lemma}
\begin{proof}
Let $\{e_1,J\bar{e}_1,{e}_{2},J\bar{e}_{2}\}$ be a local quaternionic $h$-orthonormal basis of $\Lambda^{1,0}_I(M)$. We write $\Omega=e_1\wedge J\bar{e}_1+e_2\wedge  J\bar{e}_2.$ Then a $h$-orthogonal local basis (over real-valued functions) of real forms is given by $\{\Omega,e_1\wedge J\bar{e}_1-e_2\wedge  J\bar{e}_2, e_1\wedge {e}_2+J\bar{e}_1\wedge  J\bar{e}_2,e_1\wedge J\bar{e}_2+e_2\wedge  J\bar{e}_1\}$
while a $h$-orthogonal local basis (over real-valued functions) of imaginary forms is given by $\{e_1\wedge {e}_2-J\bar{e}_1\wedge  J\bar{e}_2,e_1\wedge J\bar{e}_2-e_2\wedge  J\bar{e}_1\}$.
The Lemma follows from the definition of $\ast_\Phi.$
\end{proof}
%

We saw that, on compact $SL(2,\mathbb{H})$-manifolds, the decomposition of $(2,0)$-forms in self-dual and anti-self-dual forms passes to cohomology. In the remaining part of this Section, we will show that, on compact $SL(2,\mathbb{H})$-manifolds, the decomposition in real and imaginary forms descends to a splitting in cohomology as well.\\

For any compact hypercomplex manifold we may define the following two subgroups of $H_{\6}^{2,0}$: the real (or $\bar{J}$-invariant) subgroup
$$H_{\partial}^{\bar{J},+} := \left\{ \mathfrak{a} \in H_{\6}^{2,0}\, |\, \exists\,\varphi \in  \mathfrak{a} \text{ such that } \6 \varphi=0 \text{ and }\varphi=J\bar{\varphi}\right\}$$
and the imaginary (or $\bar{J}$-anti-invariant) subgroup
$$H_{\partial}^{\bar{J},-} := \left\{  \mathfrak{a} \in H_{\6}^{2,0}\, |\, \exists\,\varphi \in  \mathfrak{a} \text{ such that } \6 \varphi=0 \text{ and }\varphi=-J\bar{\varphi}\right\}.$$
Note that the subgroups $H_{\partial}^{\bar{J},+}$ and $H_{\partial}^{\bar{J},-}$ of $H^{2,0}_{\partial}$ are analogous to the groups $H^+_J$ and~$H^-_J$ introduced by T.-J.~Li and W.~Zhang~\cite{li-zha} in the almost-complex setting as subgroups of the second de Rham cohomology group $H^2_{dR}.$

\begin{definition}
A hypercomplex structure $(I,J,K)$ on a manifold $M$ is said to be
\begin{itemize}
\item{$C^\infty$-pure if $$H_{\partial}^{\bar{J},-} \cap H_{\partial}^{\bar{J},+} = \{0\},$$}
\item{$C^\infty$-full if $$H_{\partial}^{\bar{J},-} + H_{\partial}^{\bar{J},+} = H^{2,0}_{\partial},$$}
\item{$C^{\infty}$-pure-and-full if 
$$H_{\partial}^{\bar{J},-} \oplus H_{\partial}^{\bar{J},+} = H^{2,0}_{\partial}.$$}
\end{itemize}
\end{definition}

Imitating the proof of~\cite[Theorem 2.3]{dra-li-zha}, one can then prove the following analog of T.~Draghici, T.-J.~Li and W.~Zhang's result:

\begin{theorem}\label{pure-and-full}
Every hypercomplex structure on a compact $SL(2,\mathbb{H})$-manifold is $C^{\infty}$-pure-and-full.
\end{theorem}

\begin{remark}
There are higher-dimensional $SL(n,\mathbb{H})$-manifolds with $n>2$ whose hypercomplex structures are not $C^{\infty}$-pure-and-full: see Example 3 in Section~\ref{section11}.
\end{remark}

\section{Fr\"{o}licher degeneracy on $SL(2,\mathbb{H})$ manifolds}\label{froelicherSL}


In this Section, we show that the quaternionic Fr\"{o}licher spectral sequence always degenerates at the first page on compact $SL(2,\mathbb{H})$-manifolds: 
\begin{theorem}\label{froelicher}
On compact $SL(2,\mathbb{H})$-manifolds, the quaternionic Fr\"{o}licher spectral sequence degenerates at the first page.
\end{theorem}
This should be compared to the following well-known result:
\begin{theorem}\cite{kodaira}
On compact complex surfaces, the Fr\"{o}licher spectral sequence degenerates at the first page.
\end{theorem}



\newpage
We prove now Theorem~\ref{froelicher}:

\begin{proof}
If we show that the maps
$$H^{0,0}_\partial \xlongrightarrow{\partial_J} H^{1,0}_\partial \xlongrightarrow{\partial_J} H^{2,0}_\partial \xlongrightarrow{\partial_J} H^{3,0}_\partial \xlongrightarrow{\partial_J} H^{4,0}_\partial$$
all vanish, then we get that 
$$E^{p,0}_1 \simeq E^{p,0}_2 \quad \text{ for all } 0 \leqslant p \leqslant 3.$$
The first map vanishes because, if $\partial f =0$, then $\partial_J f=0$. Indeed, ${\bar{\6}}\bar{f}=0$ implies that $\bar{f}$ is holomorphic function on a compact manifold and hence $f$ is constant.
For the second map, take $\mathfrak{a}\in H^{2,0}_\partial$ represented by $\varphi$ such that $\varphi = \partial_J \alpha$ for some $\alpha$
representing a class in $H^{1,0}_\partial$. By the $h$-orthogonal direct sum decomposition of $H^{2,0}_{\partial}$, we know that $\mathfrak{a}$ lies either in $H^{{\Phi},+}_\partial$ or in~$H^{{\Phi},-}_\partial.$ However, any $\partial_J$-exact $\ast_\Phi$-self-dual or $\ast_\Phi$-anti-self-dual $(2,0)$-form is zero:
$$||\varphi||_h^2= ||\partial_J \alpha ||_h^2 = \int_M \partial_J \alpha \wedge \ast_{\Phi}(\partial_J \alpha) \wedge \bar{\Phi} = \pm \int_M \partial_J \alpha \wedge \partial_J \alpha \wedge \bar{\Phi} = 0.$$
%
For the third map, take $\mathfrak{a} \in H^{3,0}_\partial$ represented by~$\phi$ such that $\varphi= \partial_J \alpha$ for some $\alpha$ representing a class $\mathfrak{b}$ in
$H^{2,0}_\partial$. By Theorem~\ref{pure-and-full}, the hypercomplex structure of a compact $SL(2,\mathbb{H})$-manifold is $C^\infty$-pure-and-full.
We deduce that either $\mathfrak{b}\in H^{\bar{J},+}_\partial$ or $\mathfrak{b}\in H^{\bar{J},-}_\partial$. If $\mathfrak{b}\in H^{\bar{J},+}_\partial,$ then there exists some $\beta \in \Lambda^{2,0}_I(M) \cap \ker \partial$ such that $\beta = J\bar{\beta}$ and $\alpha = \beta + \partial \gamma$ for some $\gamma \in \Lambda^{1,0}_I(M)$. We get
\begin{eqnarray*}
\varphi &=& \partial_J \alpha = \partial_J \beta + \partial_J \partial \gamma \\
&=& \partial_J (J\bar{\beta}) - \partial \partial_J \gamma = -J(\bar{\partial \beta}) - \partial \partial_J \gamma \\
&=& - \partial (\partial_J \gamma).
\end{eqnarray*}
So $\varphi$ is $\partial$-exact and $\mathfrak{a}$ is trivial. The argument works similarly if $\mathfrak{a}\in H^{\bar{J},-}_\partial$. Finally, the fourth map vanishes because of the following:
using the pairing~(\ref{pairing}), the spaces $H^{0,0}_{\6}$ and $H^{4,0}_{\6}$ are dual on compact $SL(2,\mathbb{H})$-manifolds. Moreover, any complex-valued function $f$ satisfying $\6 f=0$
is constant. 
Now, let $\mathfrak{a}\in H^{4,0}_{\6}$ represented
by a $\6_J$-exact form $\6_J\phi$. Since $\int_Mc\,\6_J\phi\wedge\bar\Phi=0$ for any constant $c$, it follows from the pairing~(\ref{pairing}) that $\mathfrak{a}$
is the trivial class and hence the fourth map vanishes.
\end{proof}
It is natural to ask whether there are counter-examples in higher dimensions:
\begin{question}
Are there compact $SL(n,\mathbb{H})$-manifolds with $n>2$ on which the quaternionic Fr\"{o}licher spectral sequence does not degenerate at the first page?
\end{question}


\section{``non-HKT-ness" degrees}\label{section9}


%

Similarly to the non-K\"{a}hler-ness degrees introduced in~\cite{ang-dlo-tom} by Angella--Dloussky--Tomassini
on compact complex manifolds (see also~\cite{lub-tel,tel}), we define {\it{non-HKT-ness degrees}}~$\Delta^p$ on compact hypercomplex manifolds of real dimension $4n$:
$$\Delta^{p}:=h^{p,0}_{BC}+h^{p,0}_{AE}-2\dim E_2^{p,0},$$ for $0\leqslant p\leqslant 2n.$
By Theorem~\ref{froh_global}, the $\Delta^p$ are non-negative, whilst Theorem~\ref{_bott_chern_aeppli_equa_} assures that the $\partial \partial_J$-Lemma holds if and only if all $\Delta^p$ are zero.
Recall that on a compact $SL(n,\Bbb H)$ HKT manifold, we have the following result:
\begin{theorem}\cite[Theorem 2.3]{gra-lej-ver}\label{HKTimplies}
Let $M$ be a compact $SL(n,\Bbb H)$ HKT manifold. Then the $\partial \partial_J$-Lemma holds.
\end{theorem}
\begin{corollary}\label{vanish-degree}
Let $M$ be a compact $SL(n,\Bbb H)$ manifold. If $M$ is HKT then $\Delta^{p}=0$ for all $0\leqslant p\leqslant 2n.$
\end{corollary}

The goal of this Section is to show that 
a compact $SL(2,\mathbb{H})$-manifold is HKT if and only if the second non-HKT-ness degree vanishes.
This should be interpreted as a quaternionic version of the following result of Teleman~\cite[Lemma 2.3]{tel} and Angella--Dloussky--Tomassini~\cite{ang-dlo-tom} and the recent refinement in \cite{ang-tom-ver}.
\begin{theorem}\cite{tel}
A compact complex surface is K\"{a}hler if and only if the second non-K\"{a}hler-ness degree vanishes.
\end{theorem}
The proof relies on the main result of~\cite{gra-lej-ver}:
\begin{theorem}\cite{gra-lej-ver}\label{result-mehdi-geo-misha}
Let $(M,I,J,K,\Phi)$ be a compact $SL(2,\Bbb H)$-manifold. Then $M$ is HKT if and only if  $h^{1,0}_{\6}$ is even dimensional.
\end{theorem}
We also need the following results: 
\begin{theorem}\cite[Theorem 6.5]{gra-lej-ver}\label{exact_seq_AE}
Let $(M,I,J,K,\Phi)$ be a compact $SL(2,\mathbb{H})$ manifold. For a quaternionic Gauduchon metric $\Omega$, the degree map 
$$\text{deg} : H^{1,0}_{AE} \to \mathbb{C}: \alpha \mapsto \int_M \partial \alpha \wedge \Omega \wedge \bar{\Phi}$$
is well-defined and the sequence
\begin{equation*}
0 \longrightarrow{} H^{1,0}_{\partial} \longrightarrow{} H^{1,0}_{AE} \xrightarrow{\text{deg}} \mathbb{C}
\end{equation*}
is exact.
\end{theorem}


\begin{lemma}\cite[Lemma 6.3]{gra-lej-ver}\label{first-b-vanish}
Let $(M,I,J,K,\Omega)$ be a compact hypercomplex manifold of real dimension $4n$. Then $b^{1,0}=0$.
\end{lemma}
Remark that the proof of~\cite[Lemma 6.3]{gra-lej-ver} holds without the $SL(2,\mathbb{H})$ assumption and works in any real dimension $4n.$

\begin{theorem}\label{vanishing_degree}
On a compact $SL(2,\Bbb H)$-manifold, $\Delta^1=\Delta^3=0.$ 
\end{theorem}
\begin{proof}
First, by Theorem~\ref{froelicher}, $E_1^{p,0} \simeq E_2^{p,0}$ on compact $SL(2,\mathbb{H})$-manifolds. Using $SL(2,\mathbb{H})$-symmetry, we get
$$\Delta^1 = h^{1,0}_{BC} + h^{1,0}_{AE} - 2 h^{1,0}_\partial = h^{3,0}_{AE} + h^{3,0}_{BC} - 2 h^{3,0}_\partial = \Delta^3.$$
So it is sufficient to prove that $\Delta^1=0.$
Theorem~\ref{exact_seq_AE} implies that $$h^{1,0}_{\partial}\leqslant h^{1,0}_{AE} \leqslant h^{1,0}_{\partial}+1.$$  
We deduce from Lemma~\ref{first-b-vanish} that  $h^{1,0}_{BC}\leqslant  h^{1,0}_{\6}$ and from Theorem~\ref{froh_global} we get that
\begin{equation*}
2\,h^{1,0}_{\6} \leqslant h^{1,0}_{AE}+ h^{1,0}_{BC} \leqslant 2\,h^{1,0}_{\6} +1.
\end{equation*}
However, $J$ defines a quaternionic structure on both $H^{1,0}_{BC}$ and $H^{1,0}_{AE}$. Thus the dimensions $h^{1,0}_{AE}$ and $h^{1,0}_{BC}$
are even and $h^{1,0}_{AE}+h^{1,0}_{BC}=2\, h^{1,0}_{\6}$. Hence, $\Delta^1=0$.
\end{proof}

We have now all the ingredients to prove the desired result:
\begin{theorem}\label{dichotomy}
On a compact $SL(2,\mathbb{H})$-manifold $\Delta^2 \in \{0,2\}$. Moreover, the $SL(2,\mathbb{H})$-manifold is HKT if and only if $\Delta^2=0$.
\end{theorem}

\begin{proof}
By Theorem~\ref{froelicher}, $E_1^{2,0} \simeq E_2^{2,0}$ on compact $SL(2,\mathbb{H})$-manifolds. Hence
$$\Delta^2 = h_{BC}^{2,0} + h_{AE}^{2,0} - 2 h^{2,0}_{\partial} = 2 \left( h_{BC}^{2,0} - h_\partial^{2,0} \right),$$
where the last equality comes from $SL(2,\mathbb{H})$-symmetry. Moreover, on a compact hypercomplex manifold the sequence
$$ 0 \xrightarrow{} H^{1,0}_\partial \xrightarrow{} H^{1,0}_{AE} \xrightarrow{\partial} H^{2,0}_{BC} \xrightarrow{} H^{2,0}_\partial$$
is exact. We deduce that
$$h^{1,0}_\partial - h^{1,0}_{AE} + h^{2,0}_{BC} - h^{2,0}_\partial + \dim_{\mathbb{C}} \text{coker}(H^{2,0}_{BC} \to H^{2,0}_{\partial}) =0.$$
Using Theorem~\ref{froh_global} and Theorem~\ref{exact_seq_AE} we get
\begin{equation*}
0 \leqslant \Delta^2 = 2 \left( h_{BC}^{2,0} - h_\partial^{2,0} \right) \leqslant 2 \left(h^{1,0}_{AE} - h^{1,0}_\partial \right) \leqslant 2.
\end{equation*}
The first part $\Delta^2 \in \{0,2\}$ follows because $\Delta^2$ is even. For the second part, recall that Theorem~\ref{HKTimplies} assures that compact $SL(2,\mathbb{H})$ HKT implies $\Delta^2=0$. Conversely, if we assume that $\Delta^2=0$, then~$h^{2,0}_{BC} = h^{2,0}_{\partial}$. Recall that Theorem~\ref{froelicher} implies $E_1^{2,0} \simeq E_2^{2,0}$ and that we get from Theorem~\ref{vanishing_degree} that $a^{3,0}=0$. We deduce from Theorem~\ref{froh-lemma} that $b^{3,0}=0$ and from Proposition~\ref{conjugation-symmetry} that $e^{2,0}=0$. Using Varouchas' exact sequences we conclude that the map $H^{2,0}_{BC} \to H^{2,0}_{\partial}$ is surjective and thus also injective. We then deduce from the above exact sequence that $h^{1,0}_{AE} = h^{1,0}_\partial$. As $J$ defines a quaternionic structure on $H^{1,0}_{AE}$, the dimension $h^{1,0}_{AE}$ is even and we conclude from Theorem~\ref{result-mehdi-geo-misha} that $M$ is HKT.
\end{proof}

\section{Quaternionic strongly Gauduchon metrics}\label{section10}


By definition, any HKT metric is a quaternionic strongly Gauduchon metric. In this Section we show that, on compact $SL(2,\mathbb{H})$-manifolds, both notions are actually equivalent:
\begin{theorem}\label{SG-HKT}
A compact $SL(2,\mathbb{H})$-manifold is quaternionic strongly Gauduchon if and only if it is HKT.
\end{theorem}
This is a quaternionic version of the following result:
\begin{theorem}\cite[Observation 4.4]{pop}
A compact complex surface is strongly Gauduchon if and only if it is K\"{a}hler.
\end{theorem}

To imitate D.~Popovici's proof in the complex realm, we will need the following result whose proof follows from Lemma~\ref{first-b-vanish}.

\begin{lemma}\label{sequence2}
On compact hypercomplex manifolds, the sequence
\begin{equation*}
0\longrightarrow H^{1,0}_{BC}\longrightarrow H^{1,0}_{\6}\xlongrightarrow{\partial_J} H^{2,0}_{BC}
\end{equation*}
is exact. 
\end{lemma}

We have now all the tools to prove Theorem~\ref{SG-HKT}:

\begin{proof}
Suppose that $\Omega$ is a strongly Gauduchon metric such that $\6_J \Omega=\6\beta$ for some $\beta\in\Lambda_I^{2,0}(M)$.
We will prove that the map $H^{1,0}_\partial \xrightarrow{\partial_J} H^{2,0}_{BC}$ in the exact sequence of Lemma~\ref{sequence2} vanishes.
Let $\alpha$ be a representative of an element $\mathfrak{a}$ of $H^{1,0}_{\6}$:
\begin{eqnarray*}
\int_M \6_J\alpha\wedge\Omega\wedge\bar{\Phi}&=&-\int_M \alpha\wedge \6_J \Omega\wedge\bar{\Phi}\\
&=&-\int_M\alpha\wedge \6 \beta\wedge\bar{\Phi}\\
&=&\int_M \6\alpha\wedge\beta\wedge\bar{\Phi}=0.
\end{eqnarray*}
Using elliptic theory (see~\cite[Theorem 4.6]{gra-lej-ver}), there exists a real smooth function $f$
such that $\left(\6_J\alpha+\6\6_Jf\right)\wedge\Omega\wedge\bar{\Phi}=0$. This implies that we can choose
a representative of $\mathfrak{a}$ such that $\6_J \alpha \wedge \Omega \wedge \bar{\Phi} = 0$, or equivalently such that $\6_J\alpha\wedge\Omega=0.$
By a quaternionic version of the Hodge--Riemann relations~\cite{_Verbitsky:balanced_}:
\begin{equation*}
\|\6_J\alpha\|_h^2=\int_M \6_J\alpha \wedge \6 \left(J\bar\alpha\right)\wedge\bar{\Phi}=0.
\end{equation*}
Thus, the map $H^{1,0}_\partial \xrightarrow{\partial_J} H^{2,0}_{BC}$ in the exact sequence of Lemma~\ref{sequence2} vanishes. We deduce that $\dim H^{1,0}_{BC}=\dim H^{1,0}_{\6}.$ This implies that $\dim H^{1,0}_{\6}$ is even-dimensional and we conclude by Theorem~\ref{result-mehdi-geo-misha} that $M$ is HKT.
\end{proof}

\begin{remark}\label{non-existenceSG}
There are compact hypercomplex manifolds which do not admit a quaternionic strongly Gauduchon metric, see Example~1 in Section~\ref{section11}.
\end{remark}
%


\section{Examples}\label{section11}

In this Section we will compute the quaternionic Bott--Chern and Aeppli cohomologies for some left-invariant hypercomplex structures on nilmanifolds.

Let $M=\Gamma\setminus G$ be a nilmanifold with $G$ a connected nilpotent Lie group equipped with a left-invariant hypercomplex structure $(I,J,K)$ and $\Gamma$ a discrete and co-compact subgroup of $G$.
As the hypercomplex structure descends to the quotient $M=\Gamma \setminus G$, the manifold $M$ is called a \textsl{hypercomplex nilmanifold}.
It has been proven in~\cite{_BDV:nilmanifolds_} that the holonomy of the Obata connection of a hypercomplex nilmanifold of real dimension $4n$ lies in $SL(n,\mathbb{H})$.
A way to construct such hypercomplex nilmanifolds is to consider a nilpotent Lie algebra  $\mathfrak{g}$ equipped with a left-invariant hypercomplex structure $(I,J,K)$
and rational structure constants. Then, by a result of Mal'\u{c}ev~\cite{Mal}, the corresponding simply-connected nilpotent Lie group $G$ admits a lattice $\Gamma$.  

We briefly recall some known facts about Dolbeault cohomologies of complex nilmanifolds. Let $(M,I)$ be a complex nilmanifold with left-invariant complex structure $I$. Then it is conjectured that the Dolbeault cohomology $$H^{p,q}_{\bar{\partial}}(M)= \frac{ \{ \varphi \in \Lambda^{p,q}_I(M)~|~\bar\partial \varphi = 0\} }{\bar{\6} \Lambda^{p,q-1}_I(M)}$$ is isomorphic to the Dolbeault cohomology $H^{p,q}_{\bar{\partial}}(\mathfrak{g}_\mathbb{C})$ of $G$-left-invariant forms, where~$\mathfrak{g}_\mathbb{C}$ is the complexification of $\mathfrak{g}$ the Lie algebra of $G$. This isomorphism has been proven in several cases and we refer the reader
to~\cite{con-fin,cor-fer-gra-uga,rol,rol-1,sak}.
In particular, S.~Console and A.~Fino~\cite{con-fin} have proven the isomorphism when the complex structure is {\it{rational}}, i.e.~when $J\left(\mathfrak{g}_{\mathbb{Q}}\right)\subseteq \mathfrak{g}_{\mathbb{Q}},$ where $\mathfrak{g}_{\mathbb{Q}}$ is the rational structure induced by $\Gamma$ (see~\cite{con-fin}).

%
\begin{corollary}\label{iso_del}
Let $M=\Gamma\setminus G$ be a compact hypercomplex nilmanifold of real dimension~$4n$ equipped with a left-invariant hypercomplex structure $(I,J,K)$. 
Suppose that the Dolbeault cohomology $H^{0,p}_{\bar{\partial}}(M)$ of the complex nilmanifold $(M,I)$
is isomorphic to the Dolbeault cohomology $H^{0,p}_{\bar{\partial}}(\mathfrak{g}_\mathbb{C})$ of $G$-left-invariant forms, for any $0\leqslant p\leqslant 2n.$
Then we have the isomorphisms
\begin{equation*}\label{isomorphism}
H^{p,0}_{\#}(\mathfrak{g}_\mathbb{C})\,{\simeq}\, H^{p,0}_{\#}(M),
\end{equation*}
where $\# \in \{\partial, \partial_J\}$. Under the same conditions, we also have the isomorphism
$$E_2^{p,0} (\mathfrak{g}_\mathbb{C}) \simeq E_2^{p,0}(M).$$
\end{corollary}
\begin{proof}
Recall that complex conjugation gives an isomorphism between $H^{p,0}_{{\partial}}(M)$ (resp.~$H^{p,0}_{{\partial}}(\mathfrak{g}_\mathbb{C})$) and $H^{0,p}_{\bar{\partial}}(M)$
(resp.~$H^{0,p}_{\bar{\partial}}(\mathfrak{g}_\mathbb{C})$). Moreover,
Proposition~\ref{conjugation-symmetry} gives an isomorphism between $H^{p,0}_{{\partial}}(M)$ (resp.~$H^{p,0}_{{\partial}}(\mathfrak{g}_\mathbb{C})$) and $H^{p,0}_{{\partial}_J}(M)$
(resp.~$H^{p,0}_{{\partial_J}}(\mathfrak{g}_\mathbb{C})$) induced by the map $\bar{J}$.
To get the isomorphism for the quaternionic Fr\"{o}licher spectral sequence, note that the map of complexes $i : (\Lambda^{p,0}_I(\mathfrak{g}_\mathbb{C}), \partial, \partial_J) \to (\Lambda^{p,0}_I(M), \partial, \partial_J)$ induces a morphism of spectral sequences $i: (E_r^{p,0}(\mathfrak{g}_\mathbb{C}),\partial_J) \to (E_r^{p,0}(M),\partial_J)$. An isomorphism for the first page then leads to an isomorphism for all subsequent pages, see~\cite[Theorem 3.4]{mcc}.
\end{proof}

In the complex setup, D.~Angella proved that, if the De Rham and Dolbeault cohomologies of the complex nilmanifold $(M,I)$ are isomorphic to the corresponding cohomologies of $G$-left-invariant forms, then the Bott--Chern and Aeppli cohomologies of $(M,I)$ can also be computed using only $G$-left-invariant forms~\cite{ang}.
In the quaternionic setup, the proof of injectivity from~\cite[Lemma 3.6]{ang} (see also~\cite{con-fin}) can readily be adapted. We deduce that the quaternionic Bott--Chern cohomology $H^{p,0}_{BC}(\mathfrak{g}_\mathbb{C})$ can be injected in $H^{p,0}_{BC}(M)$ and similarly for the Aeppli cohomology.
The next result proves surjectivity in the special case when the hypercomplex nilmanifold has real dimension~8.
\begin{theorem}\label{iso_bott_aep}
Let $M=\Gamma\setminus G$ be a compact hypercomplex nilmanifold of real dimension~$8$ equipped with a left-invariant hypercomplex structure $(I,J,K)$. 
Suppose that the Dolbeault cohomology $H^{0,p}_{\bar{\partial}}(M)$ of the complex nilmanifold $(M,I)$
is isomorphic to the Dolbeault cohomology $H^{0,p}_{\bar{\partial}}(\mathfrak{g}_\mathbb{C})$ of $G$-left-invariant forms, for any $0\leqslant p\leqslant 4.$
Then we have the isomorphisms
\begin{equation*}\label{isomorphism}
H^{p,0}_{\#}(\mathfrak{g}_\mathbb{C})\,{\simeq}\, H^{p,0}_{\#}(M),
\end{equation*}
where $\# \in \{BC, AE\}$.
\end{theorem}


\begin{proof}
To prove the isomorphism for Bott--Chern cohomology, we use Varouchas' exact sequence
$$0 \to D^{p,0}(M) \to H^{p,0}_{BC}(M) \to H^{p,0}_\partial (M).$$
By hypothesis $H^{p,0}_\partial (\mathfrak{g}_\mathbb{C}) \simeq H^{p,0}_\partial (M)$, thus it is sufficient to prove that $D^{p,0}(M)$ can be computed using only $G$-left-invariant forms. It follows from Lemma~\ref{first-b-vanish} and Proposition~\ref{conjugation-symmetry} that $D^{1,0}(M)$, and hence also $H^{1,0}_{BC}(M)$, can be computed using only $G$-left-invariant forms (this is actually true in any real dimension $4n$). Since~$M$ is a compact $SL(2,\mathbb{H})$-manifold, Theorem~\ref{vanishing_degree} implies $f^{1,0}=f^{3,0}=a^{3,0}=0.$ 
We deduce from Varouchas' sequences (Lemma~\ref{Var_sequence}) and Corollary~\ref{iso_del} that the spaces~$E^{1,0}(M)$, $E^{3,0}(M)$ and~$B^{3,0}(M)$ can be computed using only $G$-left-invariant forms.
Therefore, using Proposition~\ref{conjugation-symmetry}, we conclude the same for $D^{2,0}(M)$, $D^{3,0}(M)$ and $D^{4,0}(M).$
Finally, we point out that $D^{0,0}(M)$ is trivial.

For the isomorphism of Aeppli cohomology, note that the Hodge star operator associated to an invariant metric
$\ast_{\Phi} : \Lambda^{p,0}_I(\mathfrak{g}_\mathbb{C}) \to \Lambda^{2n-p,0}_I(\mathfrak{g}_\mathbb{C})$ induces an isomorphism $H^{p,0}_{BC}(\mathfrak{g}_\mathbb{C}) \simeq H^{2n-p,0}_{AE}(\mathfrak{g}_\mathbb{C})$ in cohomology. 
As we have already proven the isomorphism for the quaternionic Bott--Chern cohomology, the commutative diagram
$$\begin{CD}
H^{p,0}_{AE}(\mathfrak{g}_\mathbb{C})     @> >> H^{p,0}_{AE}(M)  \\
@VV \simeq V        @VV \simeq V\\
H^{4-p,0}_{BC}(\mathfrak{g}_\mathbb{C})     @> \simeq >>  H^{4-p,0}_{BC}(M)
\end{CD}$$
gives the isomorphism for quaternionic Aeppli cohomology, see~\cite[Remark 1.40, Remark 1.41]{ang-thesis}.
\end{proof}

\begin{remark}
If, on a compact hypercomplex nilmanifold, both $H^{p,0}_{BC}(M)$ and $H^{p,0}_{\partial}(M)$ can be computed using only $G$-left-invariant forms for $0\leqslant p\leqslant 4$, then $A^{p,0}(M), B^{p,0}(M), C^{p,0}(M), D^{p,0}(M), E^{p,0}(M)$ and~$F^{p,0}(M)$ can also be computed using only $G$-left-invariant forms. Moreover, since~$H^{\bar{J},+}_{\6}$ and $H^{\bar{J},-}_{\6}$
are both subgroups of $H^{2,0}_{\partial}(M)$ then, under the same hypothesis, both can be computed using $G$-left-invariant forms, too.
\end{remark}

All the complex structures considered in the examples below are rational. Hence, we may apply Console--Fino's result~\cite{con-fin}
and therefore Corollary~\ref{iso_del} and Theorem~\ref{iso_bott_aep} to compute the dimensions of the quaternionic cohomologies of the given examples.

%


\subsection{Example 1}


The first Example we consider is a hypercomplex nilmanifold of real dimension 8. It is built out of a central extension of the quaternionic Heisenberg algebra $\mathbb{R} \times H_7$ and already appeared in~\cite{fin-gra}. We will see that it is an $SL(2,\mathbb{H})$-manifold which does not admit any HKT metric. The structure equations of the Lie algebra in question are:
$$\left\{
\begin{array}{l}
de^1 = de^2= de^3 =de^4 = de^5 =0 \\
de^6= e^1 \wedge e^2 + e^3 \wedge e^4 \\
de^7= e^1 \wedge e^3 + e^4 \wedge e^2 \\
de^8= e^1 \wedge e^4 + e^2 \wedge e^3
\end{array}
\right.$$
We consider the following hypercomplex structure:
\begin{equation*}
\begin{array}{ccccccc}
I e^1= e^2 & \quad & I e^3 = e^4 & \quad & I e^5 = e^6 & \quad & I e^7 = e^8 \\
J e^1= e^3 & \quad & J e^2 = - e^4 & \quad & J e^5 = e^7 & \quad & J e^6 = - e^8 
\end{array}
\end{equation*}
A basis of left-invariant $(1,0)$-forms (with respect to $I$) is given by:
\begin{equation*}
\begin{array}{ccccccc}
\varphi^1 = e^1 - i e^2 & \quad & \varphi^2 = e^3- i e^4 & \quad & \varphi^3 = e^5 - i e^6 & \quad & \varphi^4 = e^7 - i e^8
\end{array}
\end{equation*}
If we express the structure equations in terms of the $(1,0)$-forms we get:
\begin{equation*}
\begin{array}{ccccccc}
d\varphi^1 = 0 & \quad & d\varphi^2 = 0 & \quad & d\varphi^3 = - \frac{1}{2} \left( \varphi^{1\bar{1}} + \varphi^{2\bar{2}} \right) & \quad & d\varphi^4 = \varphi^{12}
\end{array}
\end{equation*}
where we use the notations $\varphi^{12} = \varphi^1 \wedge \varphi^2$ and $\varphi^{\bar{1}} = \bar{\left(\varphi^1\right)} = e^1 + i e^2$.
We see that 
$d\varphi^{4} \not\subseteq \Lambda^{1,1}_I$ and hence the hypercomplex structure is not {\it{abelian}}~\cite{fin-gra}. If we rewrite the structure equations in terms of the differentials $\partial$ and $\partial_J$ we get:
\begin{equation*}
\begin{array}{rcrclclcc}
\partial \varphi^1 = 0 & \quad & \partial \varphi^2 = 0 & \quad & \partial \varphi^3 = 0 & \quad & \partial \varphi^4 = \varphi^{12} \\
\partial_J \varphi^1 = 0 & \quad & \partial_J \varphi^2 = 0 & \quad & \partial_J \varphi^3 = \varphi^{12} & \quad & \partial_J \varphi^4 = 0
\end{array}
\end{equation*}
%
%
This leads to the following tables:

\begin{center}
\begin{tabular}{| c || c | c | c | c |}
\hline
$(p,0)$ & $h_\partial^{p,0}$ & $h^{p,0}_{\partial_J}$ & $h^{p,0}_{BC}$ & $h^{p,0}_{AE}$ \\
\hline
\hline
$(1,0)$ & $3$ & $3$ & $2$ & $4$\\
\hline
$(2,0)$ & $4$ & $4$ & $5$ & $5$\\
\hline
$(3,0)$ & $3$ & $3$ & $4$ & $2$\\
\hline
\end{tabular}
\end{center}

\hfill

\begin{center}
\begin{tabular}{| c || c | c | c | c | c | c |}
\hline
$(p,0)$ & $a^{p,0}$ & $b^{p,0}$ & $c^{p,0}$ & $d^{p,0}$ & $e^{p,0}$ & $f^{p,0}$ \\
\hline
\hline
$(1,0)$ & $0$ & $0$ & $1$ & $0$ & $1$ & $0$ \\
\hline
$(2,0)$ & $1$ & $1$ & $1$ & $1$ & $1$ & $1$ \\
\hline
$(3,0)$ & $0$ & $1$ & $0$ & $1$ & $0$ & $0$ \\
\hline
\end{tabular}
\end{center}
We point out the following observations.
Firstly, one checks that, in accordance with Theorem~\ref{froelicher}, the quaternionic Fr\"{o}licher spectral sequence degenerates at the first page.
Secondly, in accordance with Theorem~\ref{pure-and-full}, the hypercomplex structure is $C^\infty$-pure-and-full. Indeed, the space $H^{2,0}_\partial = \frac{\langle \varphi^{12}, \varphi^{13}, \varphi^{14}, \varphi^{23}, \varphi^{24} \rangle}{\langle \varphi^{12} \rangle}$ decomposes as a direct sum of 
$$H^{\bar{J},+}_\partial =  \frac{\langle \varphi^{12}, \varphi^{13} + \varphi^{24}, \varphi^{14} - \varphi^{23} \rangle}{\langle \varphi^{12} \rangle} \quad \text{ and } \quad H^{\bar{J},-}_\partial =  \frac{\langle \varphi^{12}, \varphi^{13} - \varphi^{24}, \varphi^{14} + \varphi^{23} \rangle}{\langle \varphi^{12} \rangle}.$$
Thirdly, as predicted by Theorem~\ref{vanishing_degree}, we have $\Delta^1=\Delta^3=0$. Lastly, the second non-HKT-ness degree is non-zero $\Delta^2=2$ and hence, by Theorem~\ref{dichotomy}, this hypercomplex nilmanifold does not admit any HKT metric.
%


\subsection{Example 2}


The second Example is also based upon the central extension of the quaternionic Lie algebra $\mathbb{R} \times H_7$ but now we consider a path of hypercomplex structures as done in~\cite{fin-gra} and~\cite{gra-lej-ver}. We end up with an $SL(2,\mathbb{H})$-manifold carrying a family $t\in (0,1)$ of hypercomplex structures which is HKT for $t=\frac{1}{2}$ but not HKT for all other values of $t$. The structure equations of the Lie algebra are the same as before:
$$\left\{
\begin{array}{l}
de^1 = de^2= de^3 =de^4 = de^5 =0 \\
de^6= e^1 \wedge e^2 + e^3 \wedge e^4 \\
de^7= e^1 \wedge e^3 + e^4 \wedge e^2 \\
de^8= e^1 \wedge e^4 + e^2 \wedge e^3
\end{array}
\right.$$
We consider the family of hypercomplex structures $(I_t,J_t,K_t)$ parametrised by $t \in (0,1)$:
\begin{equation*}
\begin{array}{ccccccc}
I_t e^1= \frac{t-1}{t} e^2 & \quad & I_t e^3 = e^4 & \quad & I_t e^5 = \frac{1}{t} e^6 & \quad & I_t e^7 = e^8 \\
J_t e^1= \frac{t-1}{t} e^3 & \quad & J_t e^2 = - e^4 & \quad & J_t e^5 = \frac{1}{t} e^7 & \quad & J_t e^6 = - e^8 
\end{array}
\end{equation*}
A basis of left-invariant $(1,0)$-forms is given by:
\begin{equation*}
\begin{array}{ccccccc}
\varphi^1 = e^1 - i \frac{t-1}{t} e^2 & \quad & \varphi^2 = e^3- i e^4 & \quad & \varphi^3 = e^5 - i \frac{1}{t} e^6 & \quad & \varphi^4 = e^7 - i e^8
\end{array}
\end{equation*}
The structure equations become:
\begin{equation*}
\begin{array}{ccccccc}
d\varphi^1 = 0 & \quad & d\varphi^2 = 0 & \quad & d\varphi^3 =  \frac{1}{2(1-t)} \varphi^{1\bar{1}} - \frac{1}{2t} \varphi^{2\bar{2}} & \quad & d\varphi^4 = \frac{2t-1}{2t-2} \varphi^{12} - \frac{1}{2t-2} \varphi^{\bar{1}2}
\end{array}
\end{equation*}
If $t =\frac{1}{2}$, then $d\varphi^{i} \subseteq \Lambda^{1,1}_I$ and the complex structure is abelian whereas otherwise it is not.
In terms of the differentials $\partial$ and $\partial_J$, the structure equations can be written as:
\begin{equation*}
\begin{array}{rcrclclcc}
\partial \varphi^1 = 0 & \quad & \partial \varphi^2 = 0 & \quad & \partial \varphi^3 = 0 & \quad & \partial \varphi^4 = \frac{2t-1}{2(t-1)} \varphi^{12} \\
\partial_J \varphi^1 = 0 & \quad & \partial_J \varphi^2 = 0 & \quad & \partial_J \varphi^3 = \frac{2t-1}{2t(t-1)}\varphi^{12} & \quad & \partial_J \varphi^4 = 0
\end{array}
\end{equation*}
We conclude: if $t\neq \frac{1}{2}$ then we get the same tables as in Example 1. On the other hand, if $t=\frac{1}{2}$ then both $\partial \varphi^4=0$ and $\partial_J \varphi^3=0$. Consequently the spaces $A^{p,0}, B^{p,0}, C^{p,0}, D^{p,0}, E^{p,0}$ and $F^{p,0}$ are trivial while
$$h^{p,0}_{\partial} = h^{p,0}_{\partial_J} = h^{p,0}_{BC}= h^{p,0}_{AE} = {4 \choose p}.$$
We deduce that, just as the HKT property, the $\partial \partial_J$-Lemma is not stable:
\begin{corollary}\cite{fin-gra}
The $\partial \partial_J$-Lemma is not stable by small hypercomplex deformations.
\end{corollary}
This differs from the complex case where the $\partial \bar{\partial}$-Lemma is stable by small complex deformations~\cite{_Angella-Tomassini_, voisin, wu}. Note that we can also use this Example to show that, even on compact $SL(2,\mathbb{H})$-manifolds, the second non-HKT-ness degree is not stable by small hypercomplex deformations, whereas on compact complex surfaces, the second non-K\"{a}hler-ness degree is a topological invariant, see~\cite[Corollary 1.2]{ang-tom-ver}.


\subsection{Example 3}


The third Example we consider is an $SL(3,\mathbb{H})$-manifold whose hypercomplex structure is neither $C^\infty$-pure nor $C^\infty$-full. The structure equations of the nilpotent Lie algebra in question are
$$\left\{
\begin{array}{ccl}
de^i &=& 0 \quad \text{ for all }1 \leqslant i \leqslant 8\\
de^9 &=& e^1 \wedge e^3  \\
de^{10} &=& e^1 \wedge e^4 \\
de^{11} &=& e^1 \wedge e^7 \\
de^{12} &=& e^1 \wedge e^8
\end{array}
\right.$$
Consider the following hypercomplex structure
\begin{equation*}
\begin{array}{lclcl}
I e^1= e^2 & \quad & I e^5 = e^6 & \quad & I e^9 = e^{10} \\
I e^3 = e^4 & \quad & I e^7 = e^8 & \quad & I e^{11} = e^{12} \\
&&&& \\
J e^1= e^5 & \quad & J e^3 = e^7 & \quad & J e^9 = e^{11} \\
J e^2 = -e^6 & \quad & J e^4 = -e^8 & \quad & J e^{10} = - e^{12}
\end{array}
\end{equation*}
A basis of left-invariant $(1,0)$-forms (with respect to $I$) is given by:
\begin{equation*}
\begin{array}{cclcl}
\varphi^1 = e^1 - i e^2 & \quad & \varphi^3 = e^5- i e^6 & \quad & \varphi^5 = e^9 - i e^{10} \\
\varphi^2 = e^3 - i e^4 & \quad & \varphi^4 = e^7- i e^8 & \quad &\varphi^6= e^{11} - i e^{12}
\end{array}
\end{equation*}
This leads to the equations
\begin{equation*}
\begin{array}{rcrclr}
\partial\varphi^1 = 0 & \quad & \partial\varphi^3 = 0 & \quad & \partial\varphi^5 = \frac{1}{2}\varphi^{12} \\
\partial\varphi^2 = 0 & \quad & \partial\varphi^4 =0 & \quad & \partial\varphi^6 = \frac{1}{2}\varphi^{14} \\
&&&&& \\
\partial_J\varphi^1 = 0 & \quad & \partial_J\varphi^3 = 0 & \quad & \partial_J\varphi^5 = \frac{1}{2}\varphi^{23} \\
\partial_J\varphi^2 = 0 & \quad & \partial_J\varphi^4 =0 & \quad & \partial_J\varphi^6 = - \frac{1}{2}\varphi^{34}  
\end{array}
\end{equation*}
%
%

\noindent One checks that the quaternionic Fr\"{o}licher spectral sequence degenerates at the first page. 
A straightforward calculation shows that
$$H^{2,0}_\partial = \frac{\langle \varphi^{12}, \varphi^{13}, \varphi^{14}, \varphi^{15}, \varphi^{16}, \varphi^{23}, \varphi^{24}, \varphi^{25}, \varphi^{26}+\varphi^{45}, \varphi^{34}, \varphi^{46} \rangle}{\langle \varphi^{12}, \varphi^{14} \rangle}$$
and that 
\begin{eqnarray*}
H^{\bar{J},+}_\partial &=& \frac{\langle \varphi^{12}, \varphi^{13}, \varphi^{14}, \varphi^{23}, \varphi^{24}, \varphi^{25}+ \varphi^{46}, \varphi^{34}\rangle}{\langle \varphi^{12}, \varphi^{14} \rangle}, \\
H^{\bar{J},-}_\partial &=& \frac{\langle \varphi^{12}, \varphi^{14}, \varphi^{23}, \varphi^{26} + \varphi^{45}, \varphi^{25} - \varphi^{46}, \varphi^{34} \rangle}{\langle \varphi^{12},\varphi^{14} \rangle}.
\end{eqnarray*}
Hence the hypercomplex structure is not $C^\infty$-pure because the space
$$H^{\bar{J},-}_\partial \cap H^{\bar{J},+}_\partial = \frac{\langle \varphi^{12}, \varphi^{14}, \varphi^{23}, \varphi^{34} \rangle}{\langle \varphi^{12}, \varphi^{14}\rangle}$$
has dimension 2. Similarly, the hypercomplex structure is not $C^\infty$-full because the space
$$\left( H^{\bar{J},-}_\partial + H^{\bar{J},+}_\partial \right)^{\perp_h} = \frac{\langle \varphi^{12}, \varphi^{14}, \varphi^{15}, \varphi^{16} \rangle}{\langle \varphi^{12},\varphi^{14} \rangle}$$
has dimension 2.

\subsection*{Funding}
This work was supported by the F.R.S.-FNRS~[16594832 to P.W.].

\subsection*{Acknowledgments}
The first named author would like to thank Dan Popovici for giving his insights on Hermitian strongly Gauduchon metrics and also thanks Andrei Teleman, Gueo Grantcharov and Carl~Tipler for useful discussions.
The second named author warmly thanks Joel Fine and Misha Verbitsky for their advice and continuous support. Both authors would like to thank the referee for his helpful comments. 

\end{document}